\documentclass[reqno,oneside,12pt]{amsart}
\usepackage{color}
\usepackage{graphicx,color}
\usepackage{amsthm,amssymb,amsfonts,amssymb}
\usepackage{hyperref}

\theoremstyle{plain}

\newtheorem{thm}{Theorem}

\newtheorem{cor}{Corollary}
\newtheorem{lem}{Lemma}
\newtheorem{prop}{Proposition}
\newtheorem{defn}{Definition}
\renewcommand{\thedefn}{\kern-3pt}
\theoremstyle{remark}
\newtheorem{rem}{Remark}

\begin{document}
\newcommand{\Real}{\mathbb{R}}
\newcommand{\eps}{\varepsilon}
\newcommand{\ye}{y_\varepsilon}
\newcommand{\yek}{y^\varepsilon_k}
\newcommand{\ue}{u_\varepsilon}
\newcommand{\ve}{v_\varepsilon}
\newcommand{\lme}{{\lambda^\varepsilon}}
\newcommand{\lmep}{{\lambda^\varepsilon_+}}
\newcommand{\lmem}{{\lambda^\varepsilon_-}}
\newcommand{\lmei}{\lambda^\varepsilon_i}
\newcommand{\mue}{\mu_\varepsilon}
\newcommand{\mued}{\mu_\varepsilon}
\newcommand{\muei}{\mu^\varepsilon_i}
\newcommand{\ome}{{\om_\varepsilon}}
\newcommand{\omed}{\omega^\varepsilon}
\newcommand{\lmen}{\lambda^\varepsilon_n}
\newcommand{\lmed}{\lambda_\varepsilon}
\newcommand{\eto}{\varepsilon\to 0}
\newcommand{\gme}{\gamma_\varepsilon}
\newcommand{\ddx}{\frac d {dx}}
\newcommand{\dxi}{\frac d {d\xi}}
\newcommand{\dfdx}[1]{\frac {d{#1}}{dx}}
\newcommand{\dfdxi}[1]{\displaystyle\frac {d{#1}}{d\xi}}
\newcommand{\ddxi}[1]{\frac {d^2{#1}}{d\xi^2}}
\renewcommand{\kappa}{\varkappa}
\newcommand{\e}{\varepsilon}
\newcommand{\Om}{\Omega}
\newcommand{\om}{\omega}
\newcommand{\lm}{\lambda}
\newcommand{\intl}{\int\limits}
\newcommand{\sml}{\sum\limits}
\newcommand{\D}{\displaystyle}
\newcommand{\tand}{\qquad\text{and}\qquad}
\newcommand{\uei}{u_{\e,\,i}}
\newcommand{\uek}{u_{\e,\,k}}
\newcommand{\Uek}{U_{\e,\,k}}
\newcommand{\omek}{\om_{\varepsilon,\,k}}
\newcommand{\omp}{\om_{\varepsilon,\,k}^\pm}
\newcommand{\uok}{u_{0,\,k}}
\newcommand{\omok}{\om_{0,\,k}}
\newcommand{\uo}{u_{0}}
\newcommand{\omo}{\om_{0}}
\newcommand{\uom}{u_-}
\newcommand{\muom}{\om^-}
\newcommand{\uop}{u_+}
\newcommand{\muop}{\om^+}
\newcommand{\Vek}{V_{\e,\,k}}
\newcommand{\bl}[1]{{\color{blue}{#1}}}
\newcommand{\mg}[1]{{\color{magenta}{#1}}}
\newcommand{\rd}[1]{{\color{red}{#1}}}
\newcommand{\marrem}[1]{\marginpar{\footnotesize\textsl{\color{black}#1}}}
\newcommand{\p}{\partial}
\newcommand{\Lo}{{\mathcal{L}_1}}
\newcommand{\Lt}{{\mathcal{L}_2}}
\newcommand{\push}{\hskip6pt}
\newcommand{\cL}{{\mathcal{L}}}
\newcommand{\cH}{{\mathcal{H}}}
\newcommand{\cA}{{\mathcal{A}}}
\newcommand{\cD}{{\mathcal{D}}}
\newcommand{\cR}{{\mathcal{R}}}
\newcommand{\cI}{{\mathcal{I}}}
\newcommand{\cW}{{\mathcal{W}}}
\newcommand{\cV}{{\mathcal{V}}}
\newcommand{\Lmen}{\Lambda_{\e,n}}
\newcommand{\Uen}{U_{\e,n}}
\newcommand{\Ven}{V_{\e,n}}
\newcommand{\betaen}{\beta_{\e,n}}
\newcommand{\hUen}{\hat{U}_{\e,n}}
\newcommand{\hlme}{\hat{\lm}^\e}
\newcommand{\PRef}{(\ref{eq1:ue})-(\ref{bc:ue})}
\newcommand{\pr}[1]{\|#1\|_1}
\newcommand{\pre}[1]{\|#1\|_{\cH_\e}}

\title[Asymptotic analysis of vibrating system]{Asymptotic analysis of vibrating system
containing  stiff-heavy and flexible-light parts
}
\authors{N. Babych\footnote[1]{University of Bath, United Kingdom, e-mail: \url{n.babych@bath.ac.uk }},
Yu. Golovaty\footnote[2]{Lviv National University, Ukraine, e-mail: \url{yu_holovaty@franko.lviv.ua}}
\date{}


\maketitle

\begin{abstract}
A model of strongly inhomogeneous  medium with simultaneous perturbation of rigidity and mass density
is studied. The medium has strongly contrasting physical characteristics in two parts
with the ratio of rigidities being proportional to a small parameter $\e$.
Additionally, the ratio of mass densities is of order $\e^{-1}$.
We investigate the asymptotic behaviour of spectrum and eigensubspaces as $\e\to 0$.
Complete asymptotic expansions of eigenvalues and eigenfunctions are
constructed and justified.

We show that the limit operator is  nonself-adjoint in general and possesses  two-dimensional Jordan cells
in spite of the singular perturbed problem is associated with a self-adjoint operator in
appropriated Hilbert space $\cL_\e$. This may happen if the metric in which the problem is self-adjoint
depends on small parameter $\e$ in a singular way.  In particular, it leads to a loss of completeness
for the eigenfunction collection. We describe how  root spaces of the limit operator approximate
eigenspaces of the perturbed operator.
\end{abstract}

\maketitle

\section*{Introduction}\noindent
We consider a model of strongly inhomogeneous medium consisting of two
nearly homogeneous components. Assuming a strong contrast of the corresponding
stiffness coefficients  ${k_1 \ll k_2}$, we get that their ratio $k_1 / k_2$ 
has a small order, which we denote by $\e$. In general, the mass densities $r_1$ and $r_2$
in two parts  could be quite different as well or could be the same.
We model this  assuming that the density ratio $r_1 / r_2$
is proportional to $\e^{-m}$. We investigate how the resonance vibrations of the medium change
if the parameter $\e$ tends to $0$. In the one-dimensional case we consider the spectral problem
\begin{gather*}
     \frac{d}{dx}\left(k_\e(x) \frac{d\ue}{dx}\right) + \lme\, r_\e(x)\ue = 0 \text{ \ in } (a,b), \quad
    \alpha_1\ue'(a)+\alpha_0\ue(a) =0,\quad \beta_1\ue'(b)+\beta_0\ue(b)=0,
\end{gather*}
where $(a,b)$ is an interval in $\mathbb{R}$ containing the origin and
\begin{equation}\label{CoeffKR}
    k_\e(x)=\begin{cases}
      \push k(x)& \text{for } x\in (a,0)\\
      \e\, \kappa(x) &\text{for }x\in (0,b),
    \end{cases}\qquad
 r_\e(x)=\begin{cases}
      \e^{-m} r(x)& \text{for }x\in (a,0)\\
      \kern16pt\rho(x)&\text{for }x\in (0,b).
    \end{cases}
\end{equation}
Here $k$, $r$ and $\kappa$, $\rho$ are smooth positive functions in intervals
$[a,0]$ and $[0,b]$ respectively. At  point $x=0$ of discontinuity of the coefficients we assume that
transmission conditions $\ue(-0) = \ue(+0)$, \  $(k\ue')(-0) = \e(\kappa\ue')(+0)$ hold.

Of course, the limit properties of spectrum depend on the power $m$ characterizing the density ratio.
Intuitively, we expect that for large values of $m$ the mass density perturbation has to be dominating
whereas for small $m$ the rigidity perturbation has to be leading.
Then it has to be at least one  critical point $m$
separating the cases. It appears to be truth exactly for $m=1$,
when the mass density perturbation is strictly inverse to the stiffness one.

This paper is devoted to the critical case $m=1$. We consider the Dirichlet problem
\begin{align}
&\left(k(x)\, \ue'\right)'\, + \e^{-1}\lme \,r(x)\,\ue = 0, \quad\qquad x\in(a,0),   \label{eq1:ue}\\
\e&\left(\kappa(x)\, \ue'\right)' +\phantom{\e^{-1}} \lme\,\rho(x)\,\ue = 0, \quad\qquad x\in(0,b),  \label{eq2:ue}\\
&\,\,\ue(-0) = \ue(+0),    \quad\,\, (k\ue')(-0) = \e\,(\kappa\ue')(+0),    \label{ic1:ue}\\
&\,\,\ue(a) =0,\qquad \ue(b)=0                \label{bc:ue}
\end{align}
and investigate the asymptotic behavior of eigenvalues $\lme$ and eigenfunctions $\ue$  as $\e\to 0$.

After a proper change of spectral parameter problem {\PRef} can be represented
as a problem with  perturbation of the transmission conditions only
(cf. the example with constant coefficients below). At first blush, the problem looks 
very simple. But the point is that the problem shows  a complicated picture of the eigenspace
bifurcation. In Section \ref{sec:Convergence} we prove  that the limit behavior of the spectrum
is described in terms of a nonself-adjoint operator that has in general  multiple eigenvalues and
two-dimensional root spaces. At the same time,   {\PRef} is associated with a self-adjoint operator in
the weighted space $\cL_\e$  with the following scalar product and  norm
\begin{equation}\label{ScalarProduct0}
    (\phi,\psi)_\eps=\e^{-1}(r \phi,\,\psi)_{L_2(a,0)}+(\rho \phi,\,\psi)_{L_2(0,b)},
    \qquad \| \phi \|_{\eps}=\sqrt{(\phi,\phi)_\eps}\,.
\end{equation}
It is obvious that for each fixed $\eps>0$ the spectrum of
(\ref{eq1:ue})-(\ref{bc:ue}) is real, discrete and simple,
$0< \lm^\e_1 < \lm^\e_2 <\dots < \lm^\e_j <\dots \to\infty$ as  $j\to\infty$ and
the corresponding  real-valued eigenfunctions $\{ u_{\e,j} \}_{j=1}^\infty$ form an orthogonal
basis in $\cL_\e$.
How  may it happen? The metric in $\cL_\e$ for which the perturbed problem is self-adjoint,
depends on small parameter $\e$ in a singular way.
In Sections \ref{sec: Asymptotic expansions}, \ref{sec: Bifurcation justification} we construct
and justify the complete asymptotic expansions of eigenvalues and eigenfunctions.
Therefore there exist pairs of closely adjacent eigenvalues $\lm^\e_j$ and $\lm^\e_{j+1}$
being the bifurcation of double limit eigenvalues. 
Although the corresponding eigenfunctions $u_{\e,j} $ and $u_{\e,j+1}$ 
remain orthogonal in $\cL_\e$ for all $\e>0$,
they make an infinitely small angle between them  in $L_2(a,b)$ with the standard metric and stick together
at the limit. In particular, it leads to the loss of completeness in $L_2(a,b)$
for the limit eigenfunction collection.
Nevertheless both $u_{\e,j} $ and $u_{\e,j+1}$ converge to the same limit,
a plane $\pi(\e)$ being the linear span of these eigenfunctions has regular
asymptotic behaviour as $\e\to 0$. In fact, a root space $\pi$ corresponding
the double eigenvalue  is the limit position of plane $\pi(\e)$
as $\e\to 0$, as is shown in Theorem \ref{TwoPlanes}.
We actually prove that  the completeness property of
the perturbed eigenfunction collection passes  into the completeness
of eigenfunctions and adjoined functions of the limit nonself-adjoint operator.

This work was motivated by \cite[Ch.8]{SanchezHubertSanchezPalencia},
where the similar problem for the Laplace operator has been considered.
The authors have handled the limit operator as the direct sum of two self-adjoint operators 
that nevertheless does not entirely explain the bifurcation picture in perturbation theory of operators.
The aim of this paper is to present more rigorous and detailed study of the case
in operator framework.

Finally, let us remark that the vibrating systems with singularly perturbed stiffness and mass density
have been considered in many papers. In the case of purely stiff models (with homogeneous
mass density), the asymptotic behavior of spectra
have been  studied in \cite{Panasenko1987} - \cite{Sanchez-Palencia1980}.
Referring to problems with purely density perturbation often
involving domain perturbations,
we mention \cite{BabychPreprynt}- \cite{Chechkin2006} 
with the latter including a broad literature overview in the area.
Spectral properties of vibrating systems with  mass entirely neglected in a subdomain were also
studied in \cite{Perez2003}, \cite{Perez2005}.
To the best of our knowledge, the first asymptotic results for the problems with simultaneous  
perturbations of mass density and stiffness appear in
\cite{GomezLoboNazarovPerez1}, \cite{GomezLoboNazarovPerez2}.

\section{Preliminaries}\label{sec2}
We demonstrate an example where eigenvalue bifurcation is calculated
explicitly. If all coefficients in \eqref{eq1:ue}, \eqref{eq2:ue} are constant we get
the eigenvalue problem
\begin{eqnarray}
&\ye'' +  \omega_\e^2\, \ye = 0, \quad x \in (a,0)\cup (0,b),&    \label{example: eq1}\\
&\ye(a) =0,\quad \ye(b) = 0, \quad
 \ye(-0) = \ye(+0), \quad \ye'(-0) = \e
\ye'(+0),&      \label{example: conditions}
\end{eqnarray}
where $\omega_\e^2=\e^{-1}\lme$.
Then each non-zero solution  can be represented by
\begin{equation*}    
\ye = \begin{cases}
A_\e \sin\ome(x-a) & \text{for }x\in (a,0),\\
B_\e \sin\ome(x-b)& \text{for }x\in (0,b),
\end{cases}
\end{equation*}
with $\ome > 0$ and $A_\e, B_\e \in \mathbb{R}$.
By virtue of \eqref{example: conditions} we have
$$
A_\e \sin\ome a - B_\e \sin\ome b =0 \qquad\text{and}\qquad
A_\e \cos\ome a - \e B_\e \cos\ome b =0.
$$
Looking for a non-zero solution of the algebraic system, we get the
characteristic equation
\begin{equation} \label{example: characteristic equation}
\cos\ome a \,\sin\ome b = \e \sin\ome a \,\cos\ome b.
\end{equation}

The latter easily gives existence of the limit $\ome \to \om$ as $\e
\to 0$ such that
\begin{equation}\label{LimitTrigEq}
    \cos\om a\, \sin\om b = 0.
\end{equation}
Moreover, the root $\om$ has to be positive.
Obviously, if we suppose, contrary to our claim, that  $\ome$ goes to $0$ as $\e\to 0$,
then \eqref{example: characteristic equation} can be written in the  equivalent form
$$
\frac{\cos\ome a \,\sin\ome b}{\cos\ome b \,\sin\ome a} = \e
$$
for sufficiently small $\e$.
A passage to the limit as $\e\to 0$ and $\ome\to 0$ leads to a contradiction,
because the left-hand side converges towards the negative number $b/a$.

If $a$ and $b$ are incommensurable number, then all roots of \eqref{LimitTrigEq} are simple.
In fact, multiple roots   exist iff
$2n|a|=(2l-1)b$ for certain natural $l$ and $n$.
Let us consider the case $a=-1$ and $b=2$.
Then the lowest positive  root $\om = \pi / 2$ of \eqref{LimitTrigEq} has multiplicity 2.
On the other hand, equation \eqref{example: characteristic equation} admits the factorization
$\left(\cos \ome-  \sqrt{\frac{\e}{2+2\e}}\right)
\left(\cos \ome+ \sqrt{\frac{\e}{2+2\e}}\right)\sin \ome =0$.
Hence the lowest eigenvalues
$\omega_{\e,1} = \frac{\pi}{2}-\arcsin\sqrt{\frac{\e}{2+2\e}}$,
$\omega_{\e,2} = \frac{\pi}{2}+\arcsin\sqrt{\frac{\e}{2+2\e}}$
are  closely adjacent and converge to the same limit $\pi/2$. The corresponding eigenfunctions $y_{\e,1}$ and $y_{\e,2}$
are defined up to a constant factor as
\begin{equation}\label{solution3Yeven:ue}
    y_{\e,j}(x)=
        \begin{cases}
            (-1)^j\sqrt{2\e/(1+\e)}  \sin\omega_{\e,j}(x+1)& \text{for } x\in(-1,0),    \\
            \phantom{(-1)^j\sqrt{2\e/(1+\e)}} \sin\omega_{\e,j}(x-2)& \text{for } x\in(0,2).
        \end{cases}
\end{equation}
We see at once that the angle in $L_2(-1,2)$  between the eigenfunctions $y_{\e,1}$ and $y_{\e,2}$ is  infinitely small as $\e\to 0$,
because  both eigenfunctions converge towards the same function
$$
y_*(x) = \begin{cases}
\kern30pt 0& \text{for }x\in (-1,0),\\
\sin\D\frac{\pi}{2} (x-2)& \text{for } x\in (0,2).
\end{cases}
$$

The point of the example is that the collection of eigenfunctions $\{ u_{\e,j} \}_{j=1}^\infty$
loses the completeness property  at the limit
on account of the double eigenvalues.
We now turn to  perturbed problem (\ref{eq1:ue})-(\ref{bc:ue}) in the general case.
To shorten formulas below, we introduce notation $I_a=(a,0)$, $I_b=(0,b)$ and
\begin{equation*}
    K(x)=\begin{cases}
      k(x)& \text{for } x\in I_a\\
      \kappa(x) &\text{for }x\in I_b,
    \end{cases}\qquad
 R(x)=\begin{cases}
      r(x)& \text{for }x\in I_a\\
      \rho(x)&\text{for }x\in I_b.
    \end{cases}
\end{equation*}
\begin{prop}    \label{prop: eigenvalues evaluation}
For each number $j\in\mathbb{N}$ eigenvalue $\lm^\e_j$  of {\PRef} is a continuous
function of $\e\in(0,1)$ and $c\,\e <\lm^\e_j\le C_j \,\e$
with constants $c$, $C _j$ being independent of $\e$.
\end{prop}

\begin{proof}
The continuity of eigenvalues with respect to the small parameter follows immediately from
the mini-max principle
\begin{equation}    \label{min-max}
    \lm^\e_j = \min_{E_j}
    \max_{\begin{smallmatrix}
        v\in E_j\\ v\not=0
        \end{smallmatrix}}
    \frac{\int_a^0k v'^2\,dx+\e\int_0^b \kappa v'^2\,dx}{\e^{-1}\int_a^0 rv^2\,dx+\int_0^b \rho v^2\,dx},
\end{equation}
where the minimum is taken over all the subspaces $E_j\subset H^1_0(a,b)$ with $\dim E_j=j$.
We consider  the eigenfunctions $v_1,\dots, v_j$
corresponding to the lowest eigenvalues $\mu_1, \dots, \mu_j$ of the problem
\begin{equation}    \label{Dirichlet problem}
        (\kappa(x) v')'+\mu\rho(x) v=0,\quad x\in I_b,\qquad v(0)= v(b)=0.
\end{equation}
Extending each $v_k$ by zero to $(a,0)$ we get that the span $\mathcal{M}$ of $v_1,\dots, v_j$
is an $j$-dimensional subspace of $H^1_0(a,b)$.
 Then
\begin{equation}\label{min-max estimation 1}
    \lm^\e_j \le
        \max_{v\in \mathcal{M}}
\frac{\int_a^0k v'^2\,dx+\e\int_0^b \kappa v'^2\,dx}{\e^{-1}\int_a^0 rv^2\,dx+\int_0^b \rho v^2\,dx}=
        \max_{v\in \mathcal{M}}
        \frac{\e\int_0^b \kappa v'^2\,dx}{\int_0^b \rho v^2\,dx}=
        \e \mu_j,
\end{equation}
which establishes the upper estimate.
Next, by the same mini-max principle
\begin{equation*}    \label{min-maxForFirst}
 \begin{aligned}
    \lm^\e_j>\lm^\e_1 = \min_{H^1_0(a,b)}
    \frac{\int_a^0k v'^2\,dx+\e\int_0^b \kappa v'^2\,dx}{\e^{-1}\int_a^0 rv^2\,dx+\int_0^b \rho v^2\,dx}\geq
     \frac{k_*}{r_*}\min_{H^1_0(a,b)}
    \frac{\int_a^0 v'^2\,dx+\e\int_0^b  v'^2\,dx}{\e^{-1}\int_a^0 v^2\,dx+\int_0^b  v^2\,dx}=
    \frac{\e k_*\,\omega_{\e,1}^2}{r_*} \geq c \e,
\end{aligned}
\end{equation*}
where $k_*=\min_{x\in (a,b)} K(x)$, $r_*=\max_{x\in (a,b)} R(x)$ and $\omega_{\e,1}^2$ is the first eigenvalue of
problem \eqref{example: eq1}-\eqref{example: conditions} with constant coefficients.
It remains to note that $\omega_{\e,1}\to \pi/2$.
\end{proof}

\section{Convergence Results and Properties of Limit Problem }\label{sec:Convergence}
Let us consider the eigenvalue problem
\begin{equation}\label{LimitPrm=1}
   \begin{cases}
    \push (K(x) u')'+\mu R(x) u=0,\qquad x\in I_a\cup I_b,\\
        \push  u(a)=0,\quad u(b)=0,\quad
      u(-0)=u(+0),\quad u'(-0)=0,
   \end{cases}
\end{equation}
that will be referred to as \emph{the limit spectral problem}. The
spectrum of (\ref{LimitPrm=1}) is discrete and real (see Th. \ref{prop: limit spectrum}
below).  We introduce the space $\cH=\{f\in H^1_0(a,b)\colon f_a\in H^2(a,0)\text{ \ and \ } f_b\in H^2(0,b) \}$,
where $f_a$ and $f_b$ are the restrictions of $f$ to intervals $I_a$ and $I_b$
resp. Problem (\ref{LimitPrm=1}) admits the variational formulation: \emph{to find $\mu\in
\mathbb{C}$ and a nontrivial $u\in \cH$ such that}
\begin{equation}\label{VarLimitPrm1}
    \int_a^b K\,u'\phi'\,dx+\kappa(0)u'(+0)\phi(0)=
    \mu \int_a^b R\,u\phi\,dx
\end{equation}
\emph{for all $\phi\in C^\infty_0(a,b)$.}
We first prove a conditional results.
\begin{prop} \label{prop: existance of limits}
Given eigenvalue $\lambda^\e$ and the corresponding eigenfunction $u_\e$ of
(\ref{eq1:ue})-(\ref{bc:ue}), if $\e^{-1}\lm^{\e}\to \mu^*$ and $u_\e\to u_*$ in
$H^2$ weakly on each intervals $I_a$, $I_b$ and $u_*$ is different from zero,
then $\mu^*$ is an eigenvalue of (\ref{LimitPrm=1}) with the eigenfunction~$u_*$.
\end{prop}

\begin{proof}
We make a change of spectral parameter $\lm^{\e}=\e\mu^\e$ in (\ref{eq1:ue})-(\ref{bc:ue}), whereat
we can reduce equation (\ref{eq2:ue}) by the first order of $\e$. Then each pair $(\mu^\e, u_\e)$ satisfies
the integral identity
\begin{equation}\label{VarPertubedPrm1}
  \int_a^b K\,u_\e'\phi'\,dx+(1-\e)\kappa(0)u_\e'(+0)\phi(0)=
    \mu_\e \int_a^b R\,u_\e\phi\,dx
\end{equation}
for all $\phi\in C^\infty_0(a,b)$. The weak convergence of $u_\e$ in $H^2(0,b)$
gives the convergence $u_\e\to u_*$ in $C^1(0,b)$, in particular, $u_\e'(+0)\to
u_*'(+0)$ as well as $u_\e'(-0)\to 0$. Moreover, the limit function $u_*$ belongs to
$\cH$, since each $u_\e$ is a continuous function at $x=0$. A passage to the limit
in (\ref{VarPertubedPrm1}) implies that pair $(\mu^*, u_*)$ satisfies identity
\eqref{VarLimitPrm1}. Recall that $u_*$ is different from zero, which completes the
proof.
\end{proof}

Before improving  the convergent results, we first compute the spectrum of
the limit problem. Let us introduce space $\mathcal{L}= L_2(r,I_a)\oplus
L_2(\rho,I_b)$, where $L_2(g,I)$ is a weighted $L_2$-space with the norm
$\|v\|=\bigl(\int_I g|v|^2\bigr)^{1/2}$. We consider two operators
\begin{align*}\label{A1A2}
&\textstyle A_1=-\frac1r \frac{d}{dx} k \frac{d}{dx}\quad\text{in}\,\,L_2(r,I_a),\quad
    && \cD(A_1)= \bigl\{u\in H^2(I_a)\colon u(a)=0,\, u'(0)=0\bigr\}, \\
&\textstyle A_2=-\frac1\rho \frac{d}{dx} \kappa \frac{d}{dx}\quad\text{in}\,\,L_2(\rho,I_b),\quad
    && \cD(A_2)=\bigl\{u\in H^2(I_b)\colon u(b)=0\bigr\}.
\end{align*}
For problem (\ref{LimitPrm=1}) we assign the matrix operator
\begin{equation*}
  \mathcal{A}=  \begin{pmatrix}
      A_1 & 0\\
      0  & A_2
    \end{pmatrix}\quad\text{in}\,\,
\cL, \quad \cD(\cA)= \bigl\{(u_1,u_2)\in \cD(A_1)\oplus\cD(A_2) \colon  u_1(0) = u_2(0)\bigr\}.
\end{equation*}
The operator $\cA$ is nonself-adjoint. Actually, it is easy to check that
\begin{equation*}
  \mathcal{A}^*=  \begin{pmatrix}
      \hat{A}_1 & 0\\
      0  & \hat{A}_2
    \end{pmatrix},\quad \cD(\cA^*)= \bigl\{(v_1,v_2)\in \cD(\hat{A}_1)\oplus\cD(\hat{A}_2)
    \colon  (kv_1')(0) = (\kappa v_2')(0)\bigr\},
\end{equation*}
where $\hat{A}_1$ is the extension of operator $A_1$ to
$\cD(\hat{A}_1)=\bigl\{u\in H^2(a,0)\colon u(a)=0\bigr\}$
and $\hat{A}_2$ is the restriction of $A_2$ to $\cD(\hat{A}_2)=\bigl\{u\in \cD(A_2)\colon u(0)=0\bigr\}$.
Let $\sigma(A)$ and $\varrho(A)$ denote the spectrum and the resolvent set of an operator $A$ respectively.
Let $\cR_\mu(A)$ denote the resolvent
$(A-\mu\mathcal{I})^{-1}$ of an operator $A$, where
$\mathcal{I}$ is the identity operator in $\cL$.

\begin{defn}
  Let $u$ be an eigenvector of $\cA$ with eigenvalue $\mu$.
 A solution $u_*$ to $(\cA-\mu\cI)u_*=u$ is called an adjoined vector of $\cA$
 (corresponding to the eigenvalue $\mu$).
\end{defn}

\begin{thm}    \label{prop: limit spectrum}
(i) $\sigma(\cA)=\sigma (A_1) \cup \sigma ( \hat{A}_2).$

\noindent
(ii) If $\mu$ belongs to  $\sigma(\cA)\setminus
\bigl(\sigma (A_1) \cap \sigma ( \hat{A}_2)\bigr)$, then $\mu$ is a simple eigenvalue.
If $\mu\in \sigma (A_1) \cap \sigma ( \hat{A}_2)$, then
  $\mu$ has multiplicity  $2$ and
the corresponding root space is generated by an eigenvector and
an adjoined  vector of $\cA$.

\noindent
(iii) The set of eigenvectors and adjoined  vectors of $\cA$ forms a complete system in $\mathcal{L}$.
\end{thm}
\begin{proof} (i)
Let us consider the equation $(\cA -\mu\cI) u =f$ for fixed $f\in \cL$. In the coordinate representation we have
$A_1\, u_1-\mu u_1=f_1$, $ A_2\, u_2-\mu u_2=f_2$.
If $\mu\not\in \sigma(A_1)$, then $u_1=\cR_\mu(A_1)f_1$.
In order to find  $u_2$ we introduce the bounded intertwining
operator  $T_\mu\colon H^2(I_a)\to H^2(I_b)$ that solves the problem
$(\kappa \psi')'+\mu\rho \psi=0$ in $I_b$, $\psi(0) = g(0)$, $\psi(b)=0$
for each $g\in H^2(I_a)$.
Note that $T_\mu$ is  a well-defined operator for all $\mu\in \varrho(\hat{A}_2)$.
Then $u_2=T_\mu\cR_\mu(A_1)f_1+\cR_\mu(\hat{A}_2)f_2$ and the resolvent of $\cA$ can be written in the form
\begin{equation}\label{ResolventA}
  \cR_\mu(\cA)=  \begin{pmatrix}
       \phantom{S_\mu}\cR_\mu(A_1) & 0\\
      T_\mu\cR_\mu(A_1)  & \cR_\mu(\hat{A}_2)
    \end{pmatrix}.
\end{equation}
From the explicit representation  of $\cR_\mu(\cA)$ it follows that
sets $\sigma(\cA)$ and $\sigma (A_1) \cup \sigma (\hat{A}_2)$ coincide.

(ii)
We suppose that $\mu\in \sigma (A_1)\setminus\sigma (\hat{A}_2)$. Then there exists an eigenvector
$U_\mu=(u_1, T_\mu u_1)$, where $u_1$ is an eigenvector of $A_1$ and,
that is the same, one is  an eigenfunction of problem
$(k\phi')'+\mu r\phi=0$ in $I_a$, $\phi(a)=\phi'(0)=0$.
Note that $\mu$ is a simple eigenvalue of the problem.
Indeed, $(\cA-\mu \cI)U_\mu=0$ follows
from the evident equality $(A_2-\mu \cI)T_\mu=0$ for all $\mu\in \varrho(\hat{A}_2)$.

Suppose now that $\mu\in \sigma (\hat{A}_2)\setminus\sigma (A_1)$.
Then operator $\cA$ has the eigenvector $V_\mu=(0,u_2)$, where $u_2$ is an eigenvector of $\hat{A}_2$.
In other words, $u_2$ is an eigenfunction of the Dirichlet problem (\ref{Dirichlet problem}).
Note that each point of $\sigma (\hat{A}_2)$ is a simple eigenvalue.
Furthermore, the first component $u_1$ must be zero, since $\mu \not\in \sigma(A_1)$.

Finally we shall show that each point of intersection $\sigma (A_1)
\cap \sigma ( \hat{A}_2)$ is an eigenvalue of algebraic multiplicity
$2$. Obviously,  vector $V_\mu=(0,u_2)$,  which appears above, is an
eigenvector of $\cA$ in this case too. Next we consider the system
\begin{equation}\label{AdjSystem}
      A_1v_1-\mu\, v_1=0,\quad
      A_2v_2-\mu\, v_2=u_2
\end{equation}
determining adjoined  vectors. If $v_1=0$, then $v_2$ must be
a solution of the boundary value problem
$(\kappa \phi')'+\mu \rho \phi=-\rho u_2$ in $I_b$,  $\phi(0)=\phi(b)=0$,
which is unsolvable. Actually, since $\mu\in \sigma (\hat{A}_2)$, by the Fredholm alternative the problem admits a solution iff
$\int_0^b\rho |u_2|^2\,dx=0$. This contradicts the fact that $u_2$
is an eigenvector of $\hat{A}_2$. Consequently we have to assume that $v_1$ is an eigenvector of $A_1$
and examine the problem
$(\kappa v_2')'+\mu \rho v_2=-\rho u_2$ in $I_b$, $v_2(0)=v_1(0)$, $v_2(b)=0$.
Here the Fredholm alternative gives the solvability condition
\begin{equation}\label{AdjointVectorCond}
    \kappa(0)u_2'(0)v_1(0)=-\int_0^b\rho\, u_2^2\,dx.
\end{equation}
We satisfy one by normalization of $v_1$, because $u_2'(0)$ is different from zero.
This condition\label{pageVV}
assures the existence of  $v_2$  and
a solution $V_\mu^*=(v_1,v_2)$ of system (\ref{AdjSystem}).
Vector $V_\mu^*$ is the adjoined  vector of $\cA$. Pair
$\{V_\mu,V_\mu^*\}$ forms a basis in the root space that corresponds to
$\mu$.

The last statement of the theorem follows from the Keldysh theorem \cite{GohbergKrein}.
\end{proof}

We investigate the limit behaviour of eigenfunctions $u_{\e,n}$
normalized by conditions
\begin{equation}\label{NormalizationL2}
    \int_a^b R(x)\,u_{\e,j}^2(x)\,dx = 1,\qquad u_{\e,j}'(b)>0.
\end{equation}
Let us enumerate the eigenvalues of operator $\cA$ in increasing order
and repeat each eigenvalue according to its multiplicity:
$\mu_1\leq\mu_2\leq\cdots\leq\mu_j\leq\cdots$.
The next statement  improves the conditional results of Proposition \ref{prop: existance of limits}.
\begin{thm}\label{thm: Convergence}  There exists a one-to-one correspondence between  the set of eigenva\-lues $\{\lambda^\e_j\}_{j=1}^\infty$
of perturbed problem \eqref{eq1:ue}-\eqref{bc:ue} and the spectrum of operator $\cA$. Namely,
$\e^{-1}{\lambda^\e_j}\to \mu_j$ as  $\e\to 0$, for each $j\in \mathbb{N}$. Furthermore, a sequence of
the corresponding eigenfunctions $u_{\e,j}$ converges in $H^1(a,b)$ towards the eigenfunction $u$
with eigenvalue $\mu_j$.
\end{thm}

\begin{proof}
For the perturbed problem {\PRef}  we assign the matrix operator in $\cL$
\begin{equation*}
\begin{aligned}
  \mathcal{A}_\e=  \begin{pmatrix}
      \hat{A}_1 & 0\\
      0  & A_2
    \end{pmatrix}, \quad \cD(\cA_\e)= \bigl\{(u_1,u_2)\in&{\ } \cD(\hat{A}_1)\oplus\cD(A_2)
    \colon \\ &u_1(0) = u_2(0),\quad (ku_1')(0) = \e(\kappa u_2')(0)\bigr\}.
\end{aligned}
\end{equation*}
Clearly, if $\mu_\e$ belongs to $\sigma(\cA_\e)$, then $\e\mu_\e$ is an eigenvalue of {\PRef}.
Let us solve the equation $(\cA_\e-\mu \cI)u=f$ for $f=(f_1,f_2)\in \cL$ and $\mu\in \varrho(\cA_\e)$.
Similarly to the previous theorem we obtain $u_1=\cR_\mu(A_1)f_1+\e S_\mu u_2$,
 $u_2=T_\mu u_1+\cR_\mu(\hat{A}_2)f_2$, where $S_\mu\colon H^2(I_b)\to H^2(I_a)$ is
 a bounded intertwining operator that solves the problem
$(k \psi')'+\mu r \psi=0$ in $I_a$, $\psi(a)=0$ and $(k\psi')(0) = (\kappa g')(0)$
for each $g\in H^2(I_b)$. This yields that
\begin{equation}\label{preRes}
    \begin{pmatrix}
      I & -\e S_\mu\\
      -T_\mu  & I
    \end{pmatrix} \begin{pmatrix} u_1\\ u_2 \end{pmatrix}=
    \begin{pmatrix}
      \cR_\mu(A_1)f_1 \\
      \cR_\mu(\hat{A}_2)f_2
    \end{pmatrix},
\end{equation}
where the matrix operator in the left-hand side is invertible as a small perturbation of the invertible one.
Letting $\e\to 0$  we can assert that
\begin{equation*}
    \cR_\mu(\cA_\e)=\begin{pmatrix}
      I & -\e S_\mu\\
      -T_\mu  & I
    \end{pmatrix}^{-1}
    \begin{pmatrix}
      \cR_\mu(A_1) & 0\\
      0 & \cR_\mu(\hat{A}_2)
    \end{pmatrix}\rightarrow
    \begin{pmatrix}
      I & 0\\
      T_\mu  & I
    \end{pmatrix}
    \begin{pmatrix}
      \cR_\mu(A_1) & 0\\
      0 & \cR_\mu(\hat{A}_2)
    \end{pmatrix}.
\end{equation*}
Hence, $\cR_\mu(\cA_\e)\to \cR_\mu(\cA)$ in the uniform operator topology as $\e\to 0$,
which establishes a number-by-number convergence of the corresponding eigenvalues \cite[Th. 3.1]{GohbergKrein}.

Next we prove existence of the limit for the eigenfunctions
under normalization condition \eqref{NormalizationL2}.
 We conclude from (\ref{VarPertubedPrm1}) that
$ 
  \int_a^b K(x)u_\e'^2(x)\,dx+(1-\e)\kappa(0)u_\e'(+0)u_\e(+0)=\mu_\e.
$ 
For each $\nu$ there exists a twice differentiable solution $\psi(x,\nu)$ of
equation $\left(\kappa v'\right)' + \nu\rho\,v = 0$ in $I_b$
that satisfies conditions $v(b)=0$, $v'(b)=1$. Moreover,
$\psi(x,\nu)$ is an analytic function with respect to the second
argument for each fixed $x$ \cite[Th.1.5]{Titchmarsh}. In particular,
$\psi(x,\mu^\e)\to \psi(x,\mu)$ in $C^2(0,b)$ as $\mu^\e\to\mu$.
Then there exits constant $\beta_\eps$ such that
$u_{\e}(x)=\beta_{\eps}\psi(x,\mu^{\e})$.
Moreover, $\beta_\eps$ is bounded as $\e\to 0$, which is due to
condition (\ref{NormalizationL2}). Therefore the values
$u_\e(+0)$ and $u_\e'(+0)$ are bounded with respect to $\e$. Consequently we have
$ 
  \int_a^b K(x)u_\e'^2(x)\,dx\leq\mu_\e+(1-\e)\kappa(0)|u_\e'(+0)u_\e(+0)|\leq M.
$ 
Then finally the sequence $\{u_\e\}_{\e>0}$ is precompact in the weak topology of
$H^1(a,b)$.  Let us consider a subsequence $u_{\e'}$ such that $u_{\e'}\to u$ in
$H^1(a,b)$ weakly. We get
$u_{\e'}(x)=\beta_{\eps'}\psi(x,\mu^{\e'})\to\beta\psi(x,\mu)=u(x)$ in $C^2(0,b)$
for certain $\beta$. Note that $\beta>0$, which is due to \eqref{NormalizationL2}. Moreover,
$u'_{\e'}(+0)\to u'(+0)$ as $\e'\to 0$. A passage to the limit in
(\ref{VarPertubedPrm1}) implies that partial weak limit $u$ satisfies the identity
\begin{equation*}\label{VarLimitPrmm}
    \int_a^b K(x)u'\phi'\,dx+\kappa(0)u'(+0)\phi(0)=
    \mu \int_a^b R(x)u\phi\,dx
\end{equation*}
for all $\phi\in C^\infty_0(a,b)$. Moreover, $u$ is different from zero, since
$\int_a^b R|u|^2 \,dx=1$.
Consequently  each weakly convergent subsequence of
$\{u_\e\}_{\e>0}$ tends to $u$, where $u$ is an eigenfunction of
(\ref{LimitPrm=1}) that corresponds to the eigenvalue $\mu$ and
satisfies conditions $\|u\|_{L^2(R,(a,b))}=1$ and $u'(b)>0$. Then
the same conclusion can be drawn for the entire sequence.
\end{proof}
\begin{rem}\label{remDontDepend}
In some cases value $\e^{-1}\lambda^\e$ doesn't actually depend on $\e$. The latter takes place if and
only if the three-points problem
\begin{equation}\label{Three Point Problem}
    \push (K(x) u')'+\mu R(x) u=0 \quad\text{for}\quad x\in   I_a\cup I_b,\qquad
        \push  u(a)= u(b)= u'(-0)=u'(+0)=0
\end{equation}
has an  eigenfunction $u$ that is \emph{continuous at} $x=0$
(for a certain eigenvalue $\mu$).
This situation is possible, for instance, in the case $a=-b$ when there exists
even eigenfunction of the Dirichlet problem on $(-b,b)$.
Then a trivial verification shows that $\lambda^\e=\e\mu$ is an eigenvalue of {\PRef}
with the eigenfunction $\ue=u$ for all $\e\in(0,1]$.

\end{rem}

\begin{cor}\label{CorH2Convergence}
Restrictions of eigenfunction $u_{\e,j}$ to the intervals $I_a$ and $I_b$
converge towards the corresponding restrictions of eigenfunction $u$
 in $H^2(a,0)$ and $H^2(0,b)$ respectively.
\end{cor}
\begin{proof}
Set $\ue=u_{\e,j}$.
We consider equation \eqref{eq1:ue}  in the form
$ 
\ue'' = -k'k^{-1} \ue' - \mue r k^{-1} \ue    
$ 
in $I_a$.
Then from Theorem \ref{thm: Convergence} we have
\begin{equation}\label{convergence of second derivatives}
\ue'' \to -k' k^{-1} u' - \mu r k^{-1} u \quad\text{in}\quad L^2(a,0),
\end{equation}
where $u$ is an eigenfunction of \eqref{LimitPrm=1}.
From \eqref{LimitPrm=1} it follows that the limit \eqref{convergence of second derivatives}
is exactly the second derivative of the limiting eigenfunction in $I_a$.
The proof for interval $I_b$ is the same.
\end{proof}


\section{Formal Asymptotic Expansions of Eigenvalues and Eigenfunctions}
\label{sec: Asymptotic expansions}

\subsection{Asymptotics of Simple Eigenvalues}
In this section we construct the complete asymptotic expansions of eigenvalues $\lm^\e$ and
eigenfunctions $u_\e$.
We begin with the examination of eigenvalues $\lm^\e_j$ for which  the limit
$\mu = \D\lim_{\e\to 0} \lm^\e_j / \e$ is a simple eigenvalue of operator $\cA$.
Clearly, $\mu$ depends on $j$, which we do not indicate for the sake of notation simplicity.
The asymptotic expansions of the eigenvalues and the corresponding eigenfunctions are represented by
\begin{gather}
    \lme  \sim  \e \,(\mu + \e \nu_1 + \cdots + \e^n \nu_n + \cdots),    \label{NCexpansionL}\\
    \ue(x)\sim
\begin{cases} \label{NCexpansionU}
y_0(x) + \e y_1(x) + \cdots + \e^n y_n(x) + \cdots & \text{for } x \in I_a,\\
z_0(x) + \e z_1(x) + \cdots + \e^n z_n(x) + \cdots & \text{for } x\in I_b,
\end{cases}
\end{gather}
where $\mu$ is an arbitrary eigenvalue of limit problem (\ref{LimitPrm=1}). Then
\begin{equation}\label{SimpleU0}
 u(x)=\begin{cases}
  y_0(x) & \text{for } x \in I_a,\\
  z_0(x) & \text{for } x\in I_b
\end{cases}
\end{equation}
is the corresponding eigenfunction of (\ref{LimitPrm=1}) as it follows
from Th. \ref{thm: Convergence}. Since in this section we treat only the simple
eigenvalues $\mu$, according to Th. \ref{prop: limit spectrum} we only consider here
two possible situations: $\mu \in \sigma (A_1) \backslash \sigma ( \hat{A}_2)$
and $\mu \in \sigma ( \hat{A}_2) \backslash \sigma (A_1)$.


\subsubsection{Case $\mu \in \sigma (A_1) \backslash \sigma ( \hat{A}_2)$}\label{subsec 411}
\label{subsec: Case mu0 in spectrum of A1}
We fix the corresponding eigenfunction $y_0$ of  ope\-rator $A_1$ such that
$
\intl_a^0 r y_0^2 \,dx=1
$
and $y_0(0) > 0$. Since $\mu$ doesn't belong to the spectrum of $\hat{A}_2$ there exists
a unique solution $z_0$ to the problem
\begin{eqnarray}\label{NC z_0}
    (\kappa z_0')' + \mu \rho z_0 = 0\quad \text{ in \ } I_b, \qquad
    z_0(0) = y_0(0),\quad z_0(b) = 0.
\end{eqnarray}
An easy computation shows that the next terms of the expansions are unique solutions to
the recurrent sequence of problems
\begin{align}
&\begin{cases}\label{SimpleProbYn}
     (k y_n')' + \mu r y_n = -\nu_n r y_0- r \sml_{j=1}^{n-1} \nu_j \, y_{n-j} \quad \text{ in \ } I_a, \\
     y_n(a) = 0,\quad (k y_n')(0) = (\kappa z_{n-1}')(0),\quad
    \int_a^0 r y_n y_0 \, dx = 0,
\end{cases}\\
&\begin{cases}\label{SimpleProbZn}
     (\kappa z_n')' + \mu \rho z_n = - \rho \sml_{j=1}^n \nu_j \, z_{n-j} \quad  \text{in \ } I_b, \\
     z_n(0) = y_n(0),\quad z_n(b) = 0
\end{cases}
\end{align}
with $\nu_n = - (\kappa z_{n-1}')(0) y_0(0)$ for $n = 1,2,\dots \,\,$.
The last formula for $\nu_n$ is obtained as the solvability condition of \eqref{SimpleProbYn}.
Note that all solutions $y_n$, $z_n$ are smooth functions.
\begin{rem}\label{remz'(0)=0}
It might happened that $z_0'(0)=0$ (cf. the proof of Th. \ref{thm: Convergence}).
In this case function $u$ defined by \eqref{SimpleU0} is exactly an eigenfunction
of the perturbed problem for each $\e\in (0,1]$.
Then the construction of asymptotics  is interrupted and we can state that there exists an eigenvalue
  $\lme=\e \mu$ for all $\e>0$. The corresponding eigenfunction
  \begin{equation*}\label{StableUe}
 u_\e(x)=\begin{cases}
  y_0(x) & \text{for } x \in I_a,\\
  z_0(x) & \text{for } x\in I_b
\end{cases}
\end{equation*}
doesn't depend on $\e$.
\end{rem}


\subsubsection{Case $\mu \in \sigma ( \hat{A}_2) \backslash \sigma (A_1)$}\label{subsec 412}
\label{subsec: Case mu0 in spectrum of A2}
This situation immediately implies $y_0 = 0$ (cf. the proof of Th. \ref{prop: limit spectrum}, part (ii)).
We fix the corresponding eigenfunction $z_0$ of $\hat{A}_2$
such that
$
\intl_0^b \rho z_0^2 \,dx=1
$
and $z'_0(0) > 0$.
A trivial verification shows that the next terms of expansions (\ref{NCexpansionU})
are the unique smooth solutions to the problems
\begin{align}
&\begin{cases}\nonumber
    (k y_n')' + \mu r y_n = - r \sml_{j=1}^{n-1} \nu_j \, y_{n-j}\quad \text{ in \ } I_a, \\
    y_n(a) = 0,\quad (k y_n')(0) = (\kappa z_{n-1}')(0),
\end{cases}
\\
&\begin{cases}\label{SimpleProbZZn}
     (\kappa z_n')' + \mu \rho z_n = - \nu_n\rho z_0 - \rho \sml_{j=1}^{n-1} \nu_j \, z_{n-j} \quad \text{ in \ } I_b, \\
     z_n(0) = y_n(0),\quad z_n(b) = 0,\quad  \int_0^b \rho z_n z_0 \, dx = 0,
\end{cases}
\end{align}
with $\nu_n = - (\kappa z'_0)(0) y_n(0)$ for $n = 1,2,\dots \,\,$. Such choice of $\nu_n$ assures the solvability of
 \eqref{SimpleProbZZn}.



\subsection{Asymptotics of Double Eigenvalues}\label{sec: Asymptotic expansions. The critical cases}\label{subsec 42}
In this subsection we treat  the case when for two  successive
eigenvalues $\lm^\e_j$ and $\lm^\e_{j+1}$ the corresponding ratios
$\e^{-1}\lm^\e_j$ and $\e^{-1}\lm^\e_{j+1}$ converge to the same
limit $\mu$. It is obvious that $\mu$ must belong to the intersection $\sigma(A_1) \cup
\sigma(\hat{A}_2)$. Let us assume that the eigenvalues and the
corresponding eigenfunctions admit expansions
\begin{gather}
    \lme  \sim  \e \,(\mu+ \sqrt{\e} \nu_1 + \e \nu_2 + \cdots),    \label{CC expansion:lme}\\
    \ue(x)\sim
\begin{cases} \label{CC expansion:ue}
\phantom{v_0(x)+} \sqrt{\e}\,w_1(x) + \e\, w_2(x) + \cdots & \text{for } x\in(a,0),\\
  v_0(x) + \sqrt{\e}\,v_1(x) + \e\, v_2(x) + \cdots & \text{for } x\in(0,b),
\end{cases}
\end{gather}
because the eigenvectors of operator $\cA$ that correspond to double eigenvalues $\mu$ have the form
 $V_\mu=(0,v_0)$ (see Th. \ref{prop: limit spectrum}).
Substituting (\ref{CC expansion:lme}), (\ref{CC expansion:ue})
into the perturbed problem we obtain
\begin{align}
&
\label{CC v0}
  \push(\kappa v_0')' + \mu \rho v_0 = 0 \quad\text{in}\quad I_b,  \qquad
  \push v_0(0) =  v_0(b)=0,
\\ &
\label{CC w0}
  \push (k w_1')' + \mu r w_1 = 0 \quad\text{in}\quad I_a, \qquad
  \push  w_1(a)= w_1'(0) = 0.
\end{align}
We fix $\mu \in \sigma(A_1) \cup \sigma(\hat{A}_2)$ and introduce the functions
\begin{equation}\label{UUstar}
   U(x)= \begin{cases}
     0&\text{for } x\in I_a\\
    v(x)& \text{for }x\in I_b
    \end{cases},
    \qquad
   U_*(x)= \begin{cases}
     w_*(x)& \text{for }x\in I_a\\
    v_*(x)& \text{for }x\in I_b
    \end{cases}
\end{equation}
that correspond to the eigenvector and adjoined  vector of $\cA$
(cf. vectors $V_\mu$ and $V_\mu^*$ in Th. \ref{prop: limit spectrum}).
Here $v$ is an eigenfunction of (\ref{CC v0}) such that $\int_0^b \rho v^2\,dx=1$, $v'(0)>0$
and adjoined vector $U_*$ is chosen such that $(U, U_*)_{L_2(R,(a,b))}=0$.
We also introduce an  eigenfunction $w$ of (\ref{CC w0}) such that
$\int_a^0 r w^2\,dx=1$ and $w(0)>0$.
It follows that  $v_0=\alpha v$ and $w_1=\beta w$ with certain  constants $\alpha$ and $\beta$.
 In addition,  $\alpha$ must be different from zero. The next problems to solve are
\begin{align}
&
\label{CC v1}
 \push  (\kappa v_1')' + \mu  \rho v_1 = -\nu_1 \alpha\rho v \quad\text{in}\quad I_b, \qquad
  \push   v_1(0) = \beta w(0),\quad v_1(b)=0,
\\ &
\label{CC w1}
\push     (k w_2')' + \mu  r w_2 = -\nu_1\beta r w  \quad\text{in}\quad I_a, \qquad
\push     w_2(a)=0, \quad  k(0) w_2'(0) = \alpha\kappa(0) v'(0).
\end{align}
In general case both problems \eqref{CC v1} and \eqref{CC w1}   are
unsolvable, since $\mu $ belongs to the spectra $\sigma(A_1)$ and
$\sigma(\hat{A}_2)$ at one time. Hence we have to apply Fredholm's alternative for both the problems.
After multiplying equations \eqref{CC w1} and \eqref{CC v1} by eigenfunctions $v$ and $w$
respectively and integrating by parts, one yields the common
solvability condition:
\begin{equation}\label{CommonSolvCond}
    \left(
      \begin{array}{cc}
        0 & \omega \\
        \omega & 0
      \end{array}
    \right)\left(
             \begin{array}{c}
               \alpha \\
               \beta
             \end{array}
           \right)=-\nu_1\left(
             \begin{array}{c}
               \alpha \\
               \beta
             \end{array}
           \right),
\end{equation}
where $\omega=(\kappa wv')(0)$ is positive.
Since the first component of vector $\gamma = (\alpha,\beta)$ must
be different from zero, $-\nu_1$ is an
eigenvalue of the matrix in \eqref{CommonSolvCond}. Therefore if
either $\nu_1=\omega$ and $\gamma = (1,-1)$ or $\nu_1=-\omega$ and
$\gamma = (1,1)$, then problems \eqref{CC v1}, \eqref{CC w1} admit
solutions. Moreover,  functions $\nu_1 w_*$ and $\nu_1 v_*$ solve problems
\eqref{CC w0} and \eqref{CC v1} respectively for both values of $\nu_1$.
Actually these problems imply immediately
$(\cA - \mu) U_* = \om U$. In other words, the first corrector
is an adjoined  vector of $\cA$ that corresponds to the eigenvector $ \om U$.
It  causes no confusion that we use the same letters $U$, $U_*$ to designate
a function of $L_2(a,b)$ and a vector in $\cL$.

Summarizing, we  formally demonstrate that there exists a pair of
closely adjacent eigenvalues $\lm^\e_j$ and $\lm^\e_{j+1}$ that
admit the asymptotic expansions
\begin{equation*}
\lm^\e_j = \e \mu  - \e^{3/2} \omega + O(\e^2),\qquad \lm^\e_{j+1}
= \e \mu  + \e^{3/2} \omega + O(\e^2), \quad \text{as } \e\to 0.
\end{equation*}
As of asymptotics of eigenfunctions we have
\begin{equation*}
    u_{\e,j}(x)=U(x)-\sqrt{\e}\,\omega U_*(x)+O(\e),\qquad
    u_{\e,j+1}(x)=U(x)+\sqrt{\e}\,\omega U_*(x)+O(\e).
\end{equation*}
These eigenfunctions subtend an infinitely small angle in
$L^2$-space as $\e\to 0$. Hence $u_{\e,j}$ and $u_{\e,j+1}$ stick
together at the limit. The latter gives rise to the loss of completeness of
the limit eigenfunction system.

Suppose that $\nu_1 = \omega$ and $\gamma = (1,-1)$. Then we will
denote by $V_1$ and $W_2$ such solutions of the problems that
$\int_0^b \rho V_1 v\,dx=0$  and  $\int_a^0 r W_2 w\,dx=0$.
We see at once that $-V_1$ and $-W_2$ are solutions of
\eqref{CC v1}, \eqref{CC w1} for $\nu_1=-\omega$ and $\gamma =
(1,1)$.

From now on we distinct two branches of expansions \eqref{CC expansion:lme}
\begin{eqnarray}
\begin{array}{rcl}\label{CC expansion:lme pm}
   \lmem & \sim & \e(\mu - \sqrt{\e}\om + \e \nu_2^- + \dots + \e^{n/2}\nu_n^- + \dots), \\
    \lmep & \sim & \e(\mu + \sqrt{\e}\om + \e \nu_2^+ + \dots + \e^{n/2}\nu_n^+ + \dots),
\end{array}
\end{eqnarray}
and the corresponding branches of \eqref{CC expansion:ue} are
\begin{eqnarray}
    \ue^\pm(x)&\sim&\left\{
    \begin{array}{ll}
        \qquad\, \mp \sqrt{\e}\,w(x) \pm \,\e \,w_2^\pm(x)+ \dots + \e^{n/2}w_n^\pm(x) \dots,
        & x \in I_a,\phantom{\int\frac NN}\\
       v(x) \pm \sqrt{\e}\,v_1^\pm(x)  + \e v_2^\pm(x) + \dots + \e^{n/2}v_n^\pm(x) \dots, & x \in I_b.
    \end{array}
        \right.
        \label{CC expansion:ue pm}
\end{eqnarray}
All coefficients are endowed with indexes $+$ or $-$ if
they depend on the choice of the sign of the first corrector $\nu_1= \pm \om$.
Note that the high order correctors in \eqref{CC expansion:lme pm}, \eqref{CC expansion:ue pm}
have to be calculated separately for both the branches. We now turn to the case $\nu_1=\om$ and find coefficients
$\nu_n^+$, $w_n^+$ and $v_n^+$. To shorten notation, we  omit  upper index "$+$" {} for a while.
Next, we see that problems \eqref{CC v1}  and \eqref{CC w1} admit  many solutions
$v_1  = V_1  + \alpha_1  v$ and $w_2  = W_2  + \beta_1  w$,
where $\alpha_1$, $\beta_1$ are constants.
These constants can be obtained from the  consistency of
problems
\begin{align}
&\begin{cases}  \label{CC v2}
   (\kappa v_2 ')' + \mu \rho v_2  =
       - \nu_1  \rho \left(  V_1  + \alpha_1  v \right) - \nu_2  \rho v , \quad x\in I_b   \\
   v_2 (0) =  W_2(0) + \beta_1 w(0),\quad  v_2 (b)=0,
\end{cases}
\\ &\begin{cases}  \label{CC w3}
    (k w_3 ')' + \mu r w_3  =
        -\nu_1  r \left(  W_2 + \beta_1 w \right)   - \nu_2  r w_1  , \quad x\in I_a \\
    w_3 (a)=0, \quad  k(0) w_3 '(0) = \kappa(0) \left(  V_1  + \alpha_1  v \right)'(0).
\end{cases}
\end{align}
The solvability conditions for problems (\ref{CC v2}) and (\ref{CC
w3}), which arrive from Fredholm's alternatives, can be represented
as a linear algebraic system
\begin{equation}\label{CommonSolvCond2}
    \left(
      \begin{array}{cc}
        \nu_1  & \omega \\
        \omega & \nu_1
      \end{array}
    \right)\left(
             \begin{array}{c}
               \alpha_1  \\
               \beta_1
             \end{array}
           \right)= \left(
             \begin{array}{c}
               (\kappa   W_2 v')(0) + \nu_2  \\
               (\kappa w  V_1 )'(0) - \nu_2
             \end{array}
           \right).
\end{equation}
The system has solution if and only if $\nu_2 =\frac{1}{2}\left(\kappa w V_1\phantom{}' -\kappa W_2 v'\right)(0)$.
After the solvability condition is satisfied, system \eqref{CommonSolvCond2} has
a partial solution $\alpha_1 = \beta_1 = \frac{1}{2\om} \left(\kappa w V_1' + \kappa W_2 v'\right)(0)$
and problems \eqref{CC v2} and \eqref{CC w3} admit  solutions $ V_2 $ and $ W_3 $ such that
$\intl_0^b \rho  V_2  v \,dx = 0$ and $\intl_a^0 r  W_3  w \,dx = 0$.
Therefore, all other solutions of \eqref{CC v2} and \eqref{CC w3}
allow the representation $v_2  =  V_2  + \alpha_2  v$ and $w_3  =  W_3  + \beta_2  w$
with real constants $\alpha_2 $, $\beta_2$.

We construct the general terms of expansions \eqref{CC expansion:lme
pm} and \eqref{CC expansion:ue pm} as solutions to the problems
\begin{align}
&\begin{cases}\label{CC vn}
   (\kappa v_n ')' + \mu \rho v_n  = -\rho \sml_{j=1}^{n} \nu_j   v_{n-j} ,
   \quad x\in I_b,\\
    v_n (0) = w_{n} (0),\quad  v_n (b)=0,
\end{cases}
\\ &
\begin{cases}\label{CC wn}
    (k w_{n+1}')' + \mu r w_{n+1}  = -
    r \sml_{j=1}^{n}\nu_j  w_{n+1-j} , \quad x\in I_a, \\
    w_{n+1} (a)=0, \quad   (k w_{n+1} )'(0) = (\kappa v_{n-1} )'(0),
\end{cases}
\intertext{with}   \label{Vn-1 and Wn}
&v_{n-1}  =  V_{n-1}  + \alpha_{n-1}  v \tand w_n  =  W_n  + \beta_{n-1}  w,
\end{align}
where $ V_{n-1} $ and $ W_n $ are solutions of the previous
problems chosen accordingly to the orthogonality conditions
$\intl_0^b \rho  V_{n-1} v \, dx = 0$ and $\intl_a^0 r  W_n  w \,dx = 0$, $n \ge 2$.
Constants $\alpha_{n-1}$ and $\beta_{n-1}$ we find from the
solvability conditions for \eqref{CC vn} and \eqref{CC wn} given by
\begin{equation}\label{CommonSolvCond n}
\kern-15pt\left(
      \begin{array}{cc}
        \nu_1  & \omega \\
        \omega & \nu_1
      \end{array}
    \right)\left(
             \begin{array}{c}
               \alpha_{n-1}  \\
               \beta_{n-1}
             \end{array}
           \right)= \left(
             \begin{array}{c}
               \left(\kappa  W_n v'\right)(0) + \sml_{j=2}^{n-1} \nu_j  \alpha_{n-j} +\nu_n \\
               \left(\kappa w V_{n-1}'\right)(0) + \sml_{j=2}^{n-1} \nu_j  \beta_{n+1-j}-\nu_n
             \end{array}
           \right).
\end{equation}
The latter has a solution if and only if $\nu_n = \frac12  \left( \kappa w  V'_{n-1}- \kappa  W_n v'\right)(0)$.
Then system \eqref{CommonSolvCond n} has a partial solution
$\alpha_{n-1} = \beta_{n-1} = \frac{1}{2\om}
\left(\kappa w  V_{n-1}' + \kappa W_n v'\right)(0) + \frac1\om \sml_{j=2}^{n-1} \nu_j \alpha_{n-j}$.
Substituting the constants into \eqref{Vn-1 and Wn} we finish the general step of  recurrent algorithm.
Hence, after coming back our natation we obtain all coefficients  $\nu_n^+$,  $v_n^+$ and  $w_n^+$ of
series \eqref{CC expansion:lme pm} and \eqref{CC expansion:ue pm}.

Similarly, we can construct the coefficients $\nu_n^-$, $v_n^-$ and $w_n^-$
of series \eqref{CC expansion:lme pm} and \eqref{CC expansion:ue pm}.
Then, by induction we get that for any natural $n$ the coefficients satisfy
relations $\nu_{n}^- =(-1)^n \,\nu_{n}^+$, $v_{n}^- = (-1)^n\, v_{n}^+$ and $w_{n}^- = (-1)^n\, w_{n}^+$.

\section{Justification of Asymptotic Expansions}\label{sec: Bifurcation justification}

Let $\cL_\e$ be he weighted $L_2$-space  with the scalar product and  norm given by \eqref{ScalarProduct0}.
We also introduce space $\cH_\e$ as the Sobolev space $H^1_0(a,b)$ with scalar product and  norm
\begin{equation}\label{ScalarProduct1}
    \langle\phi,\psi\rangle_\eps=\int_a^0 k \phi'\,\psi' \,dx+\e\int_0^b \kappa \phi'\,\psi' \,dx,
    \qquad \pre{\phi}=\sqrt{\langle\phi,\phi\rangle_\eps}.
\end{equation}
It is easily seen that
\begin{equation}\label{NormsEqui}
  c \| \phi \| \leq\| \phi \|_{\eps}\leq C\e^{-1/2} \| \phi \|, \qquad
  c \e^{1/2}\pr{\phi}  \leq \pre{\phi }\leq C \pr{\phi},
\end{equation}
where $\|\cdot\|$ and $\pr{\cdot}$ are standard norms in $L_2(a,b)$ and $H_0^1(a,b)$ respectively.

For the sake of completeness, we introduce here below the classical result on quasimodes.
Let $A$ be a self-adjoint operator in Hilbert space $H$ with domain $\mathcal{D}(A)$ and
$\sigma>0$.
\begin{defn}
We will say that pair $(\mu, u)\in \mathbb{R}\times\mathcal{D}(A)$ is a quasimode
with accuracy to $\sigma$ for operator $A$ if $\|(A-\mu I)u\|_H\leq \sigma$ and $\|u\|_H=1$.
\end{defn}
\begin{lem}[Vishik and Lyusternik]\label{LemmaVishik}
  Suppose that the spectrum of $A$ is discrete. If $(\mu, u)$ is a
quasimode of $A$ with accuracy to $\sigma$, then interval
$[\mu-\sigma,\mu+\sigma]$ contains an eigenvalue of $A$.
Furthermore, if segment $[\mu-d,\mu+d]$, $d>0$, contains one and
only one eigenvalue $\lambda$ of $A$, then $\|u-v\|_H\leq
2d^{-1}\sigma$, where $v$ is an eigenfunction of $A$ with eigenvalue
$\lambda$, $\|v\|_H=1$. \cite{VishykLiusternyk, Lazutkin}
\end{lem}

\subsection{Simple Spectrum}
We will denote by $\Lambda_{\e, n} = \e \,(\mu + \e \nu_1 + \cdots + \e^n \nu_n)$ and
\begin{gather*}
   U_{\e, n}(x)=
\begin{cases}
y_0(x) + \e y_1(x) + \cdots + \e^n y_n(x)  & \text{for } x \in I_a\\
z_0(x) + \e z_1(x) + \cdots + \e^n z_n(x) & \text{for } x\in I_b
\end{cases}
\end{gather*}
the partial sums of series  \eqref{NCexpansionL}, \eqref{NCexpansionU}.
 The perturbed problem is associated with self-adjoint operator
$ A_\e = - \frac{1}{r_\e} \frac{d}{dx} k_\e \frac{d}{dx}$ in $\cL_\e$ with
the domain
$\cD(A_\e) = \{ f \in \cH\colon (kf')(-0) = \e (\kappa f')(+0)\}$, where coefficients $k_\e$, $r_\e$ are given by
\eqref{CoeffKR} for $m=1$.


\begin{thm}\label{ThSimple1}
If $\mu_j\in \sigma (A_1) \backslash \sigma ( \hat{A}_2)$, then  eigenfunction $u_{\e,j}$
of  \eqref{eq1:ue}-\eqref{bc:ue} with eigenvalue $\lambda_j^\e$ converges in $H^1(a,b)$
towards the function
  \begin{equation*}\label{Case1LimitFunction}
 u(x)=\begin{cases}
  y(x) & \text{for } x \in I_a\\
  z(x) & \text{for } x\in I_b,
\end{cases}
\end{equation*}
where $y$ is an eigenfunction of the problem  $(ky')'+\mu ry=0$ in $I_a$,
$y(a)=y'(0)=0$
with eigenvalue $\mu_j$, and $z$ is a unique solution of the problem
$(\kappa z')'+\mu_j\, \rho z=0$ in $I_b$, $z(0)=y(0)$, $z(b)=0$.

If $z'(0)=0$, then $\lambda_j^\e=\e \mu_j$ and $u_{\e,j}=u$ for all $\e>0$. Otherwise
$\lambda^\e_j$ and $u_{\e,j}$ admit asymptotics expansions \eqref{NCexpansionL}, \eqref{NCexpansionU}
obtained in \ref{subsec 411} for $\mu=\mu_j$. Moreover, the estimates of remainder terms hold
\begin{gather}\label{EstNLambda}
    \left|\e^{-1}\lambda^\e_j-(\mu_j+\e \nu_1+\cdots+\e^n \nu_n)\right|\leq c_n \e^{n+1},\\
    \|u_{\e,j}- \vartheta_\e U_{\e, n}\|_{H^1(a,b)}\leq C_n \e^{n+1},\label{EstNU}
\end{gather}
where $\vartheta_\e$ is a normalizing multiplier with
strictly positive limit as $\e\to 0$.
\end{thm}
\begin{proof} The case $z'(0)=0$ was considered in Remarks \ref{remDontDepend} and \ref{remz'(0)=0}.
Suppose that $z'(0)\neq 0$. We first check that the the series  being constructed
in \ref{subsec: Case mu0 in spectrum of A1} give us the quasimodes with accuracy to an arbitrary order.
It follows from \eqref{SimpleProbYn}, \eqref{SimpleProbZn} that
\begin{equation}\label{AeEstim}
    \left|
     r_\e^{-1}( k_\e U_{\e, n}')'+\Lambda_{\e, n}U_{\e, n}
    \right|\leq c_n \e^{n+2}
\end{equation}
in $[a,b]$ uniformly, $U_{\e, n}(a)=U_{\e, n}(b)=0$, $U_{\e, n}(-0)=U_{\e, n}(+0)$ and
\begin{equation}\label{ResidualBeta}
    \betaen=(kU_{\e,n}')(-0) - \e (\kappa U_{\e,n}')(+0)=O(\e^{n+1}),\quad \e\to 0.
\end{equation}
Note that $U_{\e,n}$ doesn't belong to the domain of $A_\eps$ since $\betaen$ is different from zero in the
general case. Set $\phi(x)=x(\frac{x}{a}-1)$ for $x\in (a,0)$ and $\phi(x)=0$ elsewhere. Then
$V_{\e,n}=U_{\e,n}+\betaen\phi$ belongs to $\cD(A_\e)$ and a simple computation gives
$ 
    \|A_\e V_{\e, n}-\Lambda_{\e, n}V_{\e, n}\|_\e\leq c_n \e^{n+3/2}.
$ 
Hence $(\Lambda_{\e, n}, V_{\e, n}/\|V_{\e, n}\|_\e)$
is a quasimode of operator $A_\eps$ with accuracy to $c_n\e^{n+2}$ because $\|V_{\e, n}\|_\e=O(\e^{-1/2})$.
According to the Vishik-Lyusternik Lemma  there exists an eigenvalue $\lme$ of $A_\e$ such that
$|\lme-\Lambda_{\e, n}|\leq c_n \e^{n+2}$, which establishes \eqref{EstNLambda}.
Moreover, there exists an unique eigenvalue $\lme=\lambda^\e_j$ with such asymptotics by Theorem \ref{thm: Convergence}.
Next, for a certain $d>0$ segment $[\Lambda_{\e, n}-d\eps, \Lambda_{\e, n}+d\eps]$ contains one
and only one eigenvalue of $A_\e$.  Repeated application of
Lemma \ref{LemmaVishik} enables us to write
$\bigl\|\|u_\e\|_\e^{-1}\cdot u_\e- \|V_{\e, n}\|_\e^{-1}\cdot V_{\e, n}\bigr\|_\e\leq 2c_n d^{-1} \e^{n+1}$,
where $u_\e=u_{\e,j}$.
Hence, by \eqref{NormsEqui}
$$
\biggl\| u_\e- \frac{\|u_\e\|_\e}{\|V_{\e, n}\|_\e} V_{\e, n}\biggr\|_\e
\leq \frac{2c_n}{d}\|u_\e\|_\e \e^{n+1}\leq C_n \e^{n+1/2}
$$
and $\vartheta_\e =\frac{\|u_\e\|_\e}{\|V_{\e, n}\|_\e}$ converges to $1$ by Theorem \ref{thm: Convergence}.

Pair $(\lme ,u_\e)$ satisfies identity $\langle u_\e, \psi\rangle_\e=\lme (u_\e, \psi)_\e$ for all $\psi\in H_0^1(a,b)$.
Similarly, $\langle V_{\e, n}, \psi\rangle_\e= \Lambda_{\e, n}(V_{\e, n}, \psi)_\e+\alpha_\e(\psi)$, where
$|\alpha_\e(\psi)|\leq c \e^{n+1/2} \pre{\psi}$. The latter gives
\begin{equation*}
    \begin{aligned}
        \pre{u_\e-\vartheta_\e V_{\e, n}}\leq \Lambda_{\e, n} \|u_\e-\vartheta_\e V_{\e, n}\|_\e+
        |\lme-\Lambda_{\e, n}|\,\|u_\e\|_\e+|\alpha_\e(u_\e-\vartheta_\e V_{\e, n})|\\\leq
        2\mu_j \,C_n \e^{n+3/2} + c_n \|u_\e\| \,\e^{n+3/2}+  c \e^{n+1/2} \pre{u_\e-\vartheta_\e V_{\e, n}}
    \end{aligned}
\end{equation*}
and consequently $\pre{u_\e-\vartheta_\e V_{\e, n}}\leq C_n \e^{n+3/2}$.
From this and \eqref{NormsEqui} we thus get estimate \eqref{EstNU}.
\end{proof}

The same proof works  for the rest part of the simple spectrum of $\cA$.
\begin{thm}\label{ThSimple2}
If $\mu_j\in \sigma ( \hat{A}_2) \backslash \sigma (A_1)$,  then  eigenfunction $u_{\e,j}$
of  \eqref{eq1:ue}-\eqref{bc:ue} with eigenvalue $\lambda_j^\e$ converges towards function
  \begin{equation*}\label{Case2LimitFunction}
 u(x)=\begin{cases}
  \kern10pt 0 & \text{for } x \in I_a,\\
  z(x) & \text{for } x\in I_b
\end{cases}
\end{equation*}
in $H^1(a,b)$, where $z$ is an eigenfunction of the problem $(\kappa z')'+\mu\, \rho z=0$
in $I_b$, $z(0)=0$, $z(a)=0$
with eigenvalue $\mu_j$.
Moreover $\lambda^\e_j$ and $u_{\e,j}$ admit asymptotic expansions \eqref{NCexpansionL}, \eqref{NCexpansionU}
obtained in \ref{subsec 412} for $\mu=\mu_j$ with the estimates of remainder terms
\begin{gather*}
    \left|\e^{-1}\lambda^\e_j-(\mu_j+\e \nu_1+\cdots+\e^n \nu_n)\right|\leq c_n \e^{n+1},\qquad
    \|u_{\e,j}-\vartheta_\e U_{\e, n}\|_{H^1(a,b)}\leq C_n \e^{n+1}.
\end{gather*}
Here $\vartheta_\e$ is a normalizing multiplier that converges to a positive constant as $\e\to 0$.
\end{thm}

\subsection{Double Spectrum}
We introduce the partial sums of \eqref{CC expansion:lme pm}, \eqref{CC expansion:ue pm}
\begin{align}\label{eigenmodes}
 \Lmen^\pm & = \e ( \mu_j \pm \e^{1/2}\omega + \e \nu_2^\pm + \dots + \e^{n/2} \nu_n^\pm ), \\
 \Uen^\pm & = \left\{\begin{array}{ll}
\mp\e^{1/2}w + \e w_2^\pm + \dots + \e^{n/2}w_{n}^\pm& \text{for } x\in I_a\\
v+\e^{1/2} v_1^\pm + \dots + \e^{n/2}v_n^\pm& \text{for }  x\in I_b
\end{array}
\right.
\end{align}
with all coefficients constructed in Section \ref{sec: Asymptotic expansions. The critical cases} for certain double eigenvalue $\mu=\mu_j=\mu_{j+1}$.
Set $\Ven^\pm=\Uen^\pm+\betaen^\pm\phi$, where $\betaen^-$ and $\betaen^+$ are residuals in condition
\eqref{ic1:ue} for $\Uen^-$ and $\Uen^+$ respectively defined similarly as in \eqref{ResidualBeta}.
Moreover, $\betaen^\pm=O(\e^{(n+1)/2})$ as $\e\to 0$.

Analysis similar to that in the proof of Theorem \ref{ThSimple1} leads to the following result.
\begin{prop}\label{PropQuiasimodesDouble}
 The pairs $(\Lambda_{\e, n}^-, \Ven^-/\|\Ven^-\|_\e)$ and $(\Lambda_{\e, n}^+, \Ven^+/\|\Ven^+\|_\e)$
are quasimodes of operator $A_\eps$ with accuracy to $c_n\e^{n/2}$.
\end{prop}

\begin{prop}
There exist two  closely adjacent eigenvalues $\lm_\e^-$ and $\lm_\e^+$ of {\PRef}
with the asymptotics
  \begin{equation}\label{PairAsympt}
    \frac{\lm_\e^\pm}{\e}=\mu_j \pm \sqrt{\e}\omega + \e \nu_2^\pm + \dots + \e^{n/2} \nu_n^\pm+O(\e^{(n+1)/2}),
  \end{equation}
where $\mu_j$ is a double eigenvalue of operator $\cA$ and $\omega$, $\nu_k^\pm$
were defined in Sec. \ref{sec: Asymptotic expansions. The critical cases}.
\end{prop}
\begin{proof}
  From Proposition \ref{PropQuiasimodesDouble} and the Vishik-Lyusternik Lemma it follows that
there exists at least one eigenvalue of $A_\e$ in each $\e^{n/2}$-vicinity of $\Lmen^-$ and $\Lmen^+$.
Moreover, $|\lm_\e^\pm -\Lmen^\pm|\le c_n \e^{n/2}$. Evidently, eigenvalues  $\lm_\e^-$, $\lm_\e^+$ are different,
because $\Lmen^+-\Lmen^-\geq \omega\e^{3/2}$ and $\e^{n/2}$-vicinities of $\Lmen^-$ and $\Lmen^+$
don't intersect for $n>3$ and sufficient small $\e$. In fact, $|\lm_\e^+-\lm_\e^-|\geq c\e^{3/2}$
for certain positive $c$. We conclude from $|\lm_\e^\pm -\Lambda_{\e,n+3}^\pm|\le c_{n+3} \e^{(n+3)/2}$ that
\begin{equation*}
    \left|\frac{\lm_\e^\pm}{\e}-(\mu_j \pm \sqrt{\e}\omega +  \dots + \e^\frac{n}{2} \nu_n^\pm)\right|
    \leq c_{n+3} \e^{\frac{n+1}{2}}+\sum_{k=1}^3\e^{\frac{n+k}{2}}|\nu_{n+k}^\pm|\leq C_n\e^{\frac{n+1}{2}},
\end{equation*}
which establishes \eqref{PairAsympt}.
\end{proof}

We consider two planes in $L_2(a,b)$. Let $\pi$ be the root subspace that corresponds to double
eigenvalue $\mu_i$ and $\pi(\e)$ be the linear span of two eigenfunctions $u_\e^-$ and $u_\e^+$
that correspond to eigenvalues $\lm_\e^-$ and $\lm_\e^+$. These eigenfunctions as above
are normalized by \eqref{NormalizationL2}.

\begin{thm}\label{TwoPlanes}
The root subspace $\pi$ is the limit position of plane $\pi(\e)$ as $\e\to 0$, namely
  $\|P_{\pi(\e)}-P_\pi\|\to 0$, where $P_{\pi(\e)}$ and $P_\pi$ are the orthogonal projectors
  onto planes $\pi(\e)$ and $\pi$ respectively.
\end{thm}
\begin{proof}
  Nevertheless both eigenfunction $u_\e^-$ and $u_\e^+$ converge to the same limit
being the eigenfunction  of $\cA$ with eigenvalue $\mu_j$, the $\pi_\e$ has regular asymptotic behaviour as $\e\to 0$.
We choose new  $L_2(R,(a,b))$-orthogonal basis in $\pi(\e)$:
$f_{\e}=\frac{1}{2}(u_\e^++u_\e^-)$,   $g_{\e}=\frac1{2\omega\sqrt{\e}}(u_\e^+-u_\e^-)$.

By Theorem \ref{thm: Convergence} the first vector $f_{\e}$ converges in $L_2$
towards eigenfunction $U\in \pi$ given by \eqref{UUstar}. Next,
function $g_{\e}$ solves the problem
\begin{equation*}
\begin{cases}\displaystyle
  (k g_{\e}')'+ \frac{\lm_\e^+}{\e}\, r g_{\e}=
        \frac{\lm_\e^--\lm_\e^+}{2\omega\e\sqrt{\e}}\,r u_\e^- \quad\text{in }I_a,\qquad
        \displaystyle
        (\kappa g_{\e}')'+ \frac{\lm_\e^+}{\e}\, \rho g_{\e}=
        \frac{\lm_\e^--\lm_\e^+}{2\omega\e\sqrt{\e}}\,\rho u_\e^- \quad\text{in }I_b,\\
   g_{\e}(a)=0, \quad g_{\e}(b)=0, \quad
   g_{\e}(-0)=g_{\e}(+0),\quad (k g_\e')(-0) = \e(\kappa g_\e')(+0).
\end{cases}
\end{equation*}
Since $\e^{-1}\lm_\e^+\to \mu_j$, $\e^{-3/2}(\lm_\e^+-\lm_\e^-)\to 2\omega$ by \eqref{PairAsympt}
and the right-hand side is orthogonal to the eigenfunction $u_\e^+$ in $\cL_\e$,
one obtains that norms $\|g_{\e}\|_{H^2(a,0)}$ and $\|g_{\e}\|_{H^2(0,b)}$ are bounded as $\e\to 0$.
Taking into account Corollary \ref{CorH2Convergence}  we can assert
that each converging subsequence $g_{\e'}$ converges as $\e\to 0$ towards a solution of the problem
\begin{equation*}
\begin{cases}\displaystyle
  (k g ')'+ \mu_j\, r g =0\quad\text{in }I_a,\qquad
        \displaystyle
        (\kappa g ')'+ \mu_j\, \rho g =
        -\rho v \quad\text{in }I_b,\\
   g (a)=0, \quad g (b)=0, \quad
   g (-0)=g (+0),\quad g'(-0) =0,
\end{cases}
\end{equation*}
because $u_\e^-$ converges to eigenfunction $U$, which equals $v$ in $I_b$  and vanishes in $I_a$.
Hence, all partial limits of the second basis vector $g_\e$ have to be the adjoined  vectors
corresponding to the eigenvalue $\mu_j$. In fact, by orthogonality of $f_\e$ and $g_\e$ these limits
belong to the line $\{\alpha U_* \,|\,\alpha\in \mathbb{R}\}\subset \pi$ ,
which is orthogonal to $U$ (see \eqref{UUstar} for definition of $U_*$).
\end{proof}

Indeed, in previous statements $\lm_\e^-=\lambda_j^\e$, $\lm_\e^+=\lambda_{j+1}^\e$
and $u_\e^-=u_{\e,j}$, $u_\e^+=u_{\e,j+1}$,
by Theorem \ref{thm: Convergence}.
Next theorem summarizes all information on bifurcation of the double spectrum.
\begin{thm}\label{ThDouble}
Let $\mu_j\in \sigma (A_1) \cap \sigma ( \hat{A}_2)$ be a double eigenvalue with eigenfunction $U$ and
adjoined function $U_*$ given by \eqref{UUstar}, $\mu_j=\mu_{j+1}$.
Then both eigenfunction $u_{\e,j}$ and $u_{\e,j+1}$ converge to the same eigenfunction $U$  and
the difference $\frac1{\sqrt{\e}}(u_{\e,j+1}-u_{\e,j})$ converges to adjoined function $\gamma U_*$
for certain $\gamma\neq 0$.
Besides, $\lm_\e^-=\lambda_j^\e$, $\lm_\e^+=\lambda_{j+1}^\e$
and $u_{\e,j}$, $u_{\e,j+1}$ admit asymptotic expansions 
\eqref{CC expansion:lme pm}, \eqref{CC expansion:ue pm}
derived in Section \ref{subsec 42} for $\mu=\mu_j$. The estimates of remainder terms hold
\begin{gather}\label{LpmEst}
\bigl|\e^{-1}\lm_\e^\pm-\bigl(\mu_j \pm \sqrt{\e}\omega + \e \nu_2^\pm + \dots + \e^{n/2} \nu_n^\pm\bigr)\bigr|\leq c_n^\pm \e^{(n+1)/2},
\\ \label{UpmEst}
    \|u_{\e,j}- \vartheta_\e^- U_{\e, n}^-\|_{H^1(a,b)}\leq C_n^- \e^{\frac{n+1}{2}},\quad
     \|u_{\e,j+1}- \vartheta_\e^+ U_{\e, n}^+\|_{H^1(a,b)}\leq C_n^+ \e^{\frac{n+1}{2}},
\end{gather}
where $\vartheta_\e^\pm$ are normalizing
 multipliers with strictly positive limit as $\e\to 0$.
\end{thm}
\begin{proof}
It remains to prove estimates \eqref{UpmEst}. From \eqref{LpmEst} and Theorem \ref{thm: Convergence}
it may be concluded that for certain $d>0$ and $n\geq 2$ interval $[\Lmen^--d\e^{2},\Lmen^-+d\e^{2}]$ contains
eigenvalue $\lambda_j^\e$ only. In view of Prop. \ref{PropQuiasimodesDouble}  and
the Vishik-Lyusternik Lemma, we have
$$
\biggl\| u_{\e,j}- \frac{\|u_{\e,j}\|_\e}{\|V_{\e, n}^-\|_\e} V_{\e, n}^-\biggr\|_\e
\leq \frac{2c_n}{d\e^2}\|u_\e\|_\e \e^{n/2}\leq C_n \e^\frac{n-5}{2}.
$$
As in the proof of Theorem \ref{ThSimple1} we can obtain
$\| u_{\e,j}-\vartheta_\e^- U_{\e, n}^-\|_{H^1(a,b)}\leq C_n \e^{\frac{n-4}{2}}$.
Since all the coefficients of sum $U_{\e, n}^-$ are bounded in $H^1(a,b)$,
the first estimate \eqref{UpmEst} follows from the last inequality with $n$ replaced by $n+5$.
The same proof works  for $u_{\e,j+1}$.
\end{proof}

\bibliographystyle{plain}

\begin{thebibliography}{1}

\bibitem{SanchezHubertSanchezPalencia}
Sanchez Hubert J., Sanchez Palencia E.
\emph{Vibration and coupling of continuous systems. Asymptotic methods}
Berlin etc.: Springer-Verlag. xv, 421 pp., 1989.


\bibitem{Titchmarsh}
{Titchmarsh}, E.~C.
\emph{Eigenfunction expansions associated with second-order differential equations}
Oxford: Clarendon Press, 1946.

\bibitem{GohbergKrein}
Gohberg~I., Krein M.
\emph{Introduction to the Theory of Linear Nonselfadjoint Operators in Hilbert Space}
American Mathematical
Society, 1969.

\bibitem{VishykLiusternyk}
Vishik M.I., Liusternik L.A.
Regular degeneration and
boundary layer for linear differential equations with a small
parameter
\emph{Usp. Mat. Nauk}, Vol 12. 5(77), 1957. 3--122.

\bibitem{Lazutkin}
Lazutkin V. F.
Semiclassical asymptotics of eigenfunctions
\emph{Partial differential equations, V, 133--171,  Encyclopaedia Math. Sci.} 34, Springer, Berlin, 1999.


\bibitem{Panasenko1987}
Panasenko G. P.
Asymptotic behavior of the eigenvalues of elliptic equations
              with strongly varying coefficients
\emph{Trudy Sem. Petrovsk.} 1987(12), 201--217. 

\bibitem{BabychGolovaty2000}
Babych N., Golovaty Yu.
Complete WKB asymptotics of high frequency vibrations in a stiff problem
\emph{Matematychni studii} \textbf{14}(1): 59 - 72, 2001.

\bibitem{Babych2000}
Babych N.
Decomposition of low frequency eigenfunctions in stiff problem
\emph{Visn. L'viv. Univ., Ser. Mekh.-Mat.}
\textbf{58}: 97-108, 2000.

\bibitem{BabychGolovaty1999}
Yu. Golovaty and N. Babych,
On WKB asymptotic expansions of high frequency vibrations in stiff problems
\emph{Int. Conf. on Diff. Equations. Equadiff '99, Berlin, Germany.
Proc. of the conference.}
Singapore: World Scientific. Vol. 1: 103-105, 2000.


\bibitem{LoboPerez1997}
Lobo M., P{\'e}rez M. E.
High frequency vibrations in a stiff problem
\emph{Math. Methods. Appl. Sci.} 1997. Vol 7, no 2, 291-311.


\bibitem{LoboNazarovPerez2003rus}
Lobo M., Nazarov S., Peres M.
Natural oscillations of a strongly inhomogeneous elastic body.
Asymptotics of and uniform estimates for remainders
\emph{Dokl. Akad. Nauk} 389(2003)2, 173--176.



\bibitem{Sanchez-Palencia1980}
Sanchez-Palencia E.
\emph{Nonhomogeneous media and vibration theory}
Lecture Notes in Physics 127.
Springer-Verlag, Berlin, 1980, 398 pp.



\bibitem{BabychPreprynt}
Babych N.
\emph{Vibrating system contaning a part with small kinetic energy}
Preprint NASU. Centre of Math. Modelling,
Pidstryhach Institute for 
APMM;
01-2001. Lviv, 2001. 48 pp.



\bibitem{GolCR}
Golovaty Yu. D., G{\'o}mez D., Lobo M. and P{\'e}rez E.
\emph{Asymptotics for the eigenelements
of vibrating membranes with very heavy thin inclusions} C. R. Mecanique 330 (2002)
777-782.

\bibitem{GolMAS}
Golovaty Yu. D., G{\'o}mez D., Lobo M. and P{\'e}rez E.
On vibrating Membranes with very heavy thin inclusions
\emph{Math. Models Methods Appl. Sci.} Vol. 14,  no.7 (2004), 987--1034.


\bibitem{MelnykNazarov1}
Melnyk T., Nazarov. S.
The asymptotic structure of the spectrum in the problem of harmonic oscillations of a hub with heavy
spokes
{\em Dokl. Akad. Nauk of Russia}, {\bf 333} (1993) No.1, pp.13-15
(in Russian); and english translation in {\em Russian Acad. Sci. Dokl. Math.}, v.48,
1994, No.3, pp.428-432.

\bibitem{MelnykNazarov2}
Melnyk T., Nazarov. S.
Asymptotic analysis of the Neumann  problem in a junction of body and heavy spokes.
{\em Algebra i Analiz}, {\bf 12} Vol. 2(2000) pp. 188-238; and English translation in {\em
St.Petersburg Math.J.}, Vol.12, No.2 (2001), pp. 317-351.

\bibitem{Chechkin2006}
Chechkin G. A.
Homogenization of a model spectral problem for the Laplace operator
in a domain with many closely located "heavy" and
"intermediate heavy" concentrated masses.
Problems in mathematical analysis.
No. 32. \emph{J. Math. Sci.} (N. Y.) 135 (2006), no. 6, 3485--3521.

\bibitem{Perez2003}
Perez E.
On the whispering gallery modes on interfaces of membranes
composed of two materials with very different densities
\emph{Math. Models Meth. Appl. Sci.} 13 (2003), no 1, 75--98.

\bibitem{Perez2005}
Perez E.
Spectral convergence for vibrating systems containing
a part with negligible mass
\emph{Math. Methods Appl. Sci.} 28 (2005), no. 10, 1173--1200.



\bibitem{GomezLoboNazarovPerez1}
Gomez D., Lobo M., Nazarov S., Perez E.
Spectral stiff problems in domains surrounded by thin bands:
asymptotic and uniform estimates for eigenvalues
\emph{J. Math. Pures Appl.} (9):85, 2006. no. 4, 598--632.

\bibitem{GomezLoboNazarovPerez2}
Gomez D., Lobo M., Nazarov S., Perez E.
Asymptotics for the spectrum of the Wentzell problem with a small parameter
and other related stiff problems
\emph{J. Math. Pures Appl.} (9) 86 (2006), no. 5, 369--402.










\end{thebibliography}

\end{document}